\RequirePackage{fix-cm}
\documentclass[smallextended]{svjour3}       % onecolumn (second format)
\smartqed  % flush right qed marks, e.g. at end of proof
\usepackage{graphicx}

\usepackage{amsfonts}
\usepackage{amssymb}
\usepackage{amsmath}
\usepackage{mathrsfs}

\newcommand{\R}{\mathbb R}
\newcommand{\N}{\mathbb N}

\newcommand{\X}{\mathbb X}

\newcommand{\Uball}{{\mathbb B}}
\newcommand{\Usfer}{{\mathbb S}}

\newcommand{\dom}{{\rm dom}\, }

\newcommand{\nullv}{\mathbf{0}}

\newcommand{\bd}{{\rm bd}\, }
\newcommand{\inte}{{\rm int}\, }

\newcommand{\Argmin}{{\rm Argmin}}

\newcommand{\shar}{{\rm sha}}
\newcommand{\ball}[2]{{\rm B}(#1, #2)}

\newcommand{\stsl}[1]{|\nabla #1|}
\newcommand{\grsl}[1]{#1^\downarrow}
\newcommand{\Err}[1]{{\rm Er}#1}

\newcommand{\FrDer}{{\rm D}}
\newcommand{\HlDer}{{\rm D}^\downarrow_{\mathcal H}}

\newcommand{\Fam}{\Phi}
\newcommand{\Fsubd}{\partial_{\Fam}}
\newcommand{\Fsubdloc}{\partial^{\rm loc}_{\Fam}}
\newcommand{\Fsubdeps}[1]{\partial^{\rm loc}_{{\Fam}, #1}}
\newcommand{\regsubd}{\hat{\partial}}

\newcommand{\lev}[2]{[#1\le #2]}
\newcommand{\levsupp}[2]{[#1> #2]}
\newcommand{\dist}[2]{{\rm dist}\left(#1,#2\right)}
\newcommand{\Ptb}[3]{{\rm Ptb}\left(#1,#2,#3\right)}
\newcommand{\Liploc}[2]{{\rm lip}(#1,#2)}

\begin{document}

\title{On the variational behaviour of functions with positive
steepest descent rate}


\author{A. Uderzo}


\institute{A. Uderzo \at
                   Dept. of Mathematics and Applications \\
                   University of Milano-Bicocca \\
              Tel.: +39-02-64485871\\
              Fax: +39-02-64485705 \\
              \email{amos.uderzo@unimib.it} 
}

\date{\today}

\maketitle

\begin{abstract} 

This paper investigates some aspects of the variational behaviour
of nonsmooth functions, with special emphasis on certain stability
phenomena.
Relationships linking such properties as sharp minimality, superstability,
error bound and sufficiency of first-order optimality conditions are
discussed. Their study is performed by employing the steepest descent
rate, a rather general tool, which is adequate for a metric space
analysis.
The positivity of the steepest descent rate is then characterized in
terms of $\Fam$-subdifferentials. If specialized to a Banach space
setting, the resulting characterizations subsume known
results on the stability of error bounds.

\keywords{Steepest descent rate \and Strong slope \and Sharp minimizer
\and Nondifferentiable optimization \and Superstable solution \and Error bound
\and Optimality condition \and $\Fam$-subdifferential}
\subclass{49J52 \and 49J53\and 90C25\and 90C31}
\end{abstract}

\section{Introduction}
\label{intro}

It is well known that classical differential calculus provides a
powerful apparatus for a refined analysis of optimization problems.
On the other hand, successful approaches to the differentiability
of functions have revealed intriguing connections between minimality and
smoothness, with the result that the latter can be established through
variational principles. The author of the present paper shares the
opinion of all those believing that the theoretical framework emerging
from this proficuous interplay, active since several centuries, should not
exclude nonsmoothness. In fact, historically, the absence of
differentiability, when observed, was very often perceived as
a pathology (a ``miserable plague", in the Hermite's words \cite{HerSti04})
to be accurately avoided, whenever possible. Since theoretical
as well as applicative needs show that this is not always possible
(or reasonable), an area called nonsmooth analysis was
developed with the specific task to treat such a pathology,
especially for those problems arising in optimization. The present
paper is an attempt to show that nonsmoothness, along with evident
drawbacks and limitations, can also afford some benefits in the
analysis of optimization problems. This is done by considering
the favourable effects of the variational behaviour of functions
having a steepest descent rate, which is positive at some point.
The positivity of the steepest descent rate is not consistent
with the classical differentiability. In spite of this, it is the source
of several robustness phenomena having to do with perturbed
optimization: namely they relate to the local sharp minimality, the
superstability, the error bound property and its stability. A
proper general setting where to study the nature of all these
phenomena is that of metric spaces, an environment in which
it is not clear how to speak of smoothness and differentiability.
Nonetheless, nonsmooth analysis has succeeded in devising
generalized differential tools, which reveal to be adequate for
a metric space analysis.

The contents of this paper are organized as follows.
The next Section \ref{sec:2} starts with introducing the basic notion
of steepest descent rate, which will be the basic tool of analysis.
By means of that the positivity condition $(\mathcal{C})$ is
formulated. Such a condition could be regarded as a manifestation
of nonsmoothness in metric spaces. In the subsections included
in Section \ref{sec:2} condition $(\mathcal{C})$ is shown to be equivalent
to local sharp minimality, to superstability of a solution to a perturbed
optimization problem and, to a certain extent, to the local error bound.
What is more, in the presence of $(\mathcal{C})$ the last property
turns out to be stable with respect to perturbations with controlled
strong slope. Section \ref{sec:3} is devoted to the characterization
of the positivity condition $(\mathcal{C})$ in terms of global and local
$\Fam$-subdifferentials. In Section \ref{sec:4} the findings of the
previous section are specialized to a Banach space setting, where
widely employed nonsmooth analysis tools such as the Hadamard
generalized derivative, the regular subdifferential, the subdifferential
in the sense of convex analysis can be utilized.
Relationships with existing results from the related literature are
discussed. A final section is reserved to distil the spirit of
the analysis here exposed.

%%%%%%%%%%%%%%%%%%%%%%%%%%%%%%%%%%%%%%%%%%%%%%%%%%%%%%%
%%%%%%%%%%%%%%%%%%%%%%%%%%%%%%%%%%%%%%%%%%%%%%%%%%%%%%%

\section{Steepest descent rate and variational analysis in metric
spaces}
\label{sec:2}

Whenever $r\in\R\cup\{\pm\infty\}$, symbol $[r]_+$ stands for
$\max\{r,0\}$.
Given a metric space $(X,d)$, the closed ball with center $x\in X$ and
radius $r\ge 0$ is denoted by $\ball{x}{r}$. If $x\in X$ and $S
\subseteq X$, the distance of $x$ from $S$ is indicated by
$\dist{x}{S}=\inf_{y\in S}d(x,y)$, with the convention that
$\dist{x}{\varnothing}=+\infty$. Given a function
$f:X\longrightarrow\R\cup\{\pm\infty\}$, its domain is denoted by
$\dom f=\{x\in X:\ |f(x)|<\infty\}$. If $\alpha\in\R$, $\lev{f}{\alpha}
=\{x\in X:\ f(x)\le\alpha\}$ and $\levsupp{f}{\alpha}=X\backslash
\lev{f}{\alpha}$ indicate the $\alpha$-sublevel set
and the strict $\alpha$-superlevel set of $f$, respectively.
In particular, whenever it is $\inf_X f>-\infty$, $\Argmin(f)=
\lev{f}{\inf_Xf}$ denotes the set of all global minimizers
of $f$, if any.
Throughout the paper, the acronym l.s.c. stands for lower
semicontinuous.

The analysis of the variational properties of functions in metric
spaces will be mainly conducted by means of the following basic
tool.

\begin{definition}     \label{def:grsl}
Given a function $f:X\longrightarrow\R\cup\{\pm\infty\}$ defined
on a metric space $X$ and an element $\bar x\in\dom f$, the value
$$
   \grsl{f}(\bar x)=\liminf_{x\to\bar x}\frac{f(x)-f(\bar x)}{d(x,\bar x)},
$$
is called the {\em steepest descent rate} of $f$ at $\bar x$.
\end{definition}

To the best of the author's knowledge, the first employment of
the steepest descent rate in connection with extremum problems
in metric spaces goes back to \cite{Mari82}, where the notion of
inf-stationary point is introduced. Later on it found relevant
applications in nonsmooth analysis (see, for instance, \cite{Gian89}).
By his side, 
V.F. Demyanov contributed to popularize the use of this tool as well as of
its $k$-th order version: in several of his works he employed it for
formulating optimality conditions in metric spaces, as a
starting point for further developments in nondifferentiable optimization
(see, for instance, \cite{Demy00,Demy05,Demy05b,Demy10}, wherefrom
the notation of Definition \ref{def:grsl} has been borrowed). Further
recent employments can be found in \cite{Zasl13,Zasl14}.
Clearly the steepest
descent rate is strictly connected with another tool, widely utilized in
metric space variational analysis, known as strong slope. Given a
function $f:X\longrightarrow\R\cup\{\pm\infty\}$ and an element
$\bar x\in\dom f$, according to \cite{DeMaTo80}, by the strong slope
(or calmness rate) of $f$ at $\bar x$ the value
$$
   \stsl{f}(\bar x)=\left\{ \begin{array}{ll}
                       0, & \qquad\text{if $\bar x$ is a local minimizer for $f$},  \\
                       \displaystyle\limsup_{x\to\bar x} 
                       \frac{f(\bar x)-f(x)}{d(x,\bar x)}, & 
                       \qquad\text{otherwise} 
                 \end{array}
   \right.
$$
is meant. Now, it is readily seen that if $\bar x\in\dom f$ is a local minimizer of $f$,
then $\grsl{f}(\bar x)\ge 0$. Such a condition is evidently only necessary
for the local minimality of $\bar x$. Nevertheless, its enforcement
\begin{equation*}
    \grsl{f}(\bar x)>0    \leqno (\mathcal{C})
\end{equation*}
is a sufficient condition for local (strict) optimality (see \cite{Demy00}).
One of the aims of this article is to show that actually condition
$(\mathcal{C})$ can tell even more than that. Notice that,
whenever $(\mathcal{C})$ holds, it has to be
$\stsl{f}(\bar x)=0$. In circumstances
in which the annihilating of the strong slope does not allow to guarantee
those benefits deriving from nondegeneracy conditions, condition
$(\mathcal{C})$ turns out to provide meaningful insights into the variational
behaviour of $f$ near $\bar x$. As it will be illustrated in Sect. \ref{sec:4},
when functions are defined in more structured spaces, the occurence
of $(\mathcal{C})$ is essentially connected with the nonsmoothness
of $f$.

%%%%%%%%%%%%%%%%%%%%%%%%%%%%%%%%%%%%%%%%%%%%%%%%%%%

\subsection{Sharp minimizers and their superstability}
\label{subsec:2.1}

The next definition, originally introduced in \cite{Poly87}
in its global form within the context
of nondifferentiable convex optimization, captures
a possible variational behaviour of a function near a
local minimizer of it. It describes how the local minimum value
is attained at that point.

\begin{definition}      \label{def:locsharp}
Given a function  $f:X\longrightarrow\R\cup\{\pm\infty\}$, an
element $\bar x\in\dom f$ is said to be a {\it local sharp minimizer}
of $f$ if there exist positive $\sigma$ and $r$ such
that
\begin{equation}    \label{in:defsharmin}
    f(x)\ge f(\bar x)+\sigma d(x,\bar x),\quad\forall x\in
    \ball{\bar x}{r}.
\end{equation}
The value
$$
    \shar(f,\bar x)=\sup\{\sigma>0:\ \exists r>0 \hbox{ satisfying }
   (\ref{in:defsharmin})\}
$$
will be called {\it modulus of local sharpness} of $f$ at $\bar x$.
If inequality $(\ref{in:defsharmin})$ continues being true with
$\ball{\bar x}{r}$ replaced by $X$, then $\bar x$ is called
{\it global sharp minimizer} of $f$.
\end{definition}

\begin{example}
Here some simple situations are illustrated in which the notion
of sharp minimality naturally arises.

(i) Let $(X,d)$ be a complete metric space and let $T:X\longrightarrow
X$ be a contraction mapping, i.e. there exists $\alpha\in [0,1)$
such that $d(T(x_1),T(x_2))\le\alpha d(x_1,x_2)$, for every $x_1,\,
x_2\in X$. The Banach-Caccioppoli fixed point theorem ensures
the existence of a unique fixed point $\bar x\in X$, such that
$$
    d(x,\bar x)\le {1\over 1-\alpha}d(x,T(x)),\quad\forall x\in X.
$$
This inequality shows that the displacement function $f_T:X\longrightarrow
[0,+\infty)$ of $T$, defined as $f_T(x)=d(x,T(x))$, admits $\bar x$ as a global
sharp minimizer, with $\sigma=1-\alpha$.

(ii) Whenever $f:X\longrightarrow\R\cup\{\pm\infty\}$ admits $\bar x\in X$
as a local sharp minimizer and $\tilde f:X\longrightarrow\R\cup\{\pm\infty\}$
is such that
$$
    \tilde f(\bar x)=f(\bar x)\qquad\hbox{ and }\qquad \tilde f(x)\ge
    f(x),\quad\forall x\in \ball{\bar x}{r},
$$
for some $r>0$, $\bar x$ is a local sharp minimizer also for $\tilde f$.
So, if $X=\R^n$ and $\|\cdot\|$ stands for the Euclidean norm,
suppose that $f:\R^n\longrightarrow\R$ is any radial function
with profile $\pi_f:[0,+\infty)\longrightarrow[0,+\infty)$ satisfying
the inequality
$$
   \pi_f(t)\ge\sigma t,\quad\forall t\in [0,r]
$$
for some $r,\, \sigma>0$, and let $\bar x\in\R^n$. Then, every function
$\tilde f:\R^n\longrightarrow\R$ such that
$$
   \tilde f(x)\ge f(x-\bar x)+c,\quad\forall x\in\ball{\bar x}{r},
$$
with $c\in\R$, admits $\bar x$ as a local sharp minimizer.
Moreover, it results in $\shar(f,\bar x)\ge\sigma$. 
\end{example}

Notice that a local sharp minimizer is a strict minimizer of $f$
and an isolated point of the sublevel set $\lev{f}{f(\bar x)}$. As an immediate
consequence of Definition \ref{def:locsharp}, one obtains the
following local form of Tykhonov well-posedness: if $(x_n)_{n\in\N}$
is a sequence in $X$, whose elements lie sufficiently near $\bar x$,
and $f(x_n)\longrightarrow f(\bar x)$ as $n\to\infty$, then it
must be $x_n\longrightarrow\bar x$, as $n\to\infty$. Of course,
if $\bar x$ is a global sharp minimizer of $f$, the extremum problem
$\min_{x\in\X}f$ is Tykhonov well-posed. 

It was remarked already in \cite{Poly87} that condition $(\ref{in:defsharmin})$
can not be satisfied a priori by smooth functions, if considered
in a properly structured setting. In spite of this, sharp minimality
ensures good properties. For instance, it has been shown that
sharp minimality is a sufficient condition for finite termination of
the proximal point algorithm (see \cite{Rock76,Ferr91}).
A weaker version of the notion sharp minimality (known as weak
sharp minimality) gained an even major success, due to its recognized
relevance in the convergence analysis of algorithms for solving
extremum problems as well as in the study of stability of variational
problems (see \cite{BurFer93,BurDen02}).

Local sharpness of minimizers can be easily characterized
is terms of positivity of the steepest descent rate, as follows.

\begin{proposition}      \label{pro:Ccondshamin}
Given a function  $f:X\longrightarrow\R\cup\{\pm\infty\}$, an
element $\bar x\in\dom f$ is a local sharp minimizer iff
condition $(\mathcal{C})$ holds. Moreover $\grsl{f}(\bar x)=
\shar(f,\bar x)$.
\end{proposition}

\begin{proof}
The proof of the first assertion is a straightforward consequence
of Definition \ref{def:grsl} and Definition \ref{def:locsharp}.
To see the inequality $\grsl{f}(\bar x)\ge\shar(\bar x)$, fix an
arbitrary $\sigma>0$ such that $(\ref{in:defsharmin})$ holds for
some $r>0$. One has
$$
    \grsl{f}(\bar x)=\sup_{\delta>0}\inf_{x\in\ball{\bar x}{\delta}
    \backslash\{\bar x\}}{f(x)-f(\bar x)\over d(x,\bar x)}\ge
    \inf_{x\in\ball{\bar x}{r}\backslash\{\bar x\}}
    {f(x)-f(\bar x)\over d(x,\bar x)}\ge\sigma.
$$
On the other hand, by Definition \ref{def:grsl}, for any
arbitrary $\epsilon>0$ there exists $r_\epsilon>0$ such that
$$
    \inf_{x\in\ball{\bar x}{r_\epsilon}\backslash\{\bar x\}}
    {f(x)-f(\bar x)\over d(x,\bar x)}>\grsl{f}(\bar x)-\epsilon,
$$
so inequality $(\ref{in:defsharmin})$  is satisfied by $\sigma=
\grsl{f}(\bar x)-\epsilon$ for $r=r_\epsilon$. If it were 
$\grsl{f}(\bar x)>\shar(f,\bar x)$, by taking $\epsilon<
\grsl{f}(\bar x)-\shar(f,\bar x)$ this would contradict the definition
of $\shar(f,\bar x)$.
\hfill$\square$
\end{proof}

The variational behaviour characterized by condition $(\mathcal{C})$
yields favorable stability properties of solutions in perturbed
optimization. Let us consider indeed the following family of perturbed 
problems:
$$
    \min_{x\in X}[f(x)+g(x)]   \leqno (\mathcal{P}_g)
$$
where $g\in \mathcal{G}_{\bar x}$ and
$$
    \mathcal{G}_{\bar x}=\{g:X\longrightarrow\R\cup\{\pm\infty\}:\ 
    \bar x\in\dom g\}.
$$
Notice that the additive perturbation term $g$ allows one to
cover very general perturbation effects. Clearly, when $g\equiv 0$
one gets the unperturbed problem $\min_{x\in X}f(x)$.
The next definition generalizes a strong concept of stability
in optimization, which was proposed again in \cite{Poly87}.

\begin{definition}        \label{def:supersta}
With reference to a family of perturbed problems $(\mathcal{P}_g)$,
a local solution $\bar x\in\dom f$ of $(\mathcal{P}_0)$ is called
{\it superstable} (for $(\mathcal{P}_g)$) if there exists $\epsilon_0>0$ such
that $\bar x$ locally solves $(\mathcal{P}_g)$, for every $g\in
\mathcal{G}_{\bar x}$, with $\stsl{g}(\bar x)<\epsilon_0$.
\end{definition}

\begin{remark}
As a comment to Definition \ref{def:supersta},
observe that condition $\stsl{g}(\bar x)<\epsilon_0$ holds in
particular, whenever
$g\in\mathcal{G}_{\bar x}$ is locally Lipschitz around $\bar x$,
with Lipschitz constant 
$$
    \Liploc{g}{\bar x}=\limsup_{x_1,x_2\to\bar x\atop x_1\ne x_2}
   {|g(x_1)-g(x_2)|\over d(x_1,x_2)}<\epsilon_0.
$$
Thus, if $\bar x$ is superstable, it persists as a solution to $(\mathcal{P}_g)$
under locally Lipschitz perturbations of $f$. More generally, if $\bar x$ is
superstable it locally solves any problem $\min_{x\in X}\tilde f(x)$,
for every $\tilde f\in\Ptb{f}{\bar x}{\epsilon_0}$, where
$$
   \Ptb{f}{\bar x}{\epsilon_0}=\left\{\tilde f\in\mathcal{G}_{\bar x}:\ 
   \limsup_{x\to\bar x}{|\tilde f(x)-f(x)-(\tilde f(\bar x)-f(\bar x))|
   \over d(x,\bar x)}\le\epsilon_0\right\}.
$$
Indeed, if $\tilde f\in\Ptb{f}{\bar x}{\epsilon_0}$, then
$$
    \stsl{(\tilde f-f)}(\bar x)\le\limsup_{x\to\bar x}
   {|\tilde f(x)-f(x)-(\tilde f(\bar x)-f(\bar x))|
   \over d(x,\bar x)}.
$$
The above kind of perturbations has been already considered in
\cite{KrVaTh10} and will be again employed here in a subsequent
section.
\end{remark}

The next proposition reveals that the superstability property,
as presented in Definition \ref{def:supersta},
is actually a reformulation of the local sharp minimality.

\begin{proposition}      \label{pro:shamincharms}
Let $f:X\longrightarrow\R\cup\{\pm\infty\}$ be a given function
and let $\bar x\in\dom f$. Then, $\bar x$ is a local sharp
minimizer iff it is superstable for $(\mathcal{P}_g)$.
\end{proposition}

\begin{proof}
Let us start with supposing $\bar x$ to be a local sharp minimizer
of $f$. Then condition $(\mathcal{C})$ does hold. So, take
$\epsilon_0=\grsl f(\bar x)$. For any $g\in
\mathcal{G}_{\bar x}$, with $\stsl{g}(\bar x)<\epsilon_0$ one
obtains
\begin{eqnarray*}
    \grsl{(f+g)}(\bar x)&\ge & \liminf_{x\to\bar x} \frac{f(x)-
     f(\bar x)}{d(x,\bar x)}+\liminf_{x\to\bar x}\frac{g(x)-
     g(\bar x)}{d(x,\bar x)} \ge \grsl{f}(\bar x)-\stsl{g}(\bar x) \\
    &>&0.
\end{eqnarray*}
Since $(\mathcal{C})$ is a sufficient optimality condition,
this implies that $\bar x$ is also a local minimizer of $f+g$, for
every $g\in\mathcal{G}_{\bar x}$, with $\stsl{g}(\bar x)<\epsilon_0$.

Suppose now that $\bar x$ satisfies Definition  \ref{def:supersta}
with some $\epsilon_0>0$. Choose $\epsilon\in (0,\epsilon_0)$
and observe that the function $g:X\longrightarrow\R$ defined
by $g(x)=-\epsilon d(x,\bar x)$ belongs to $\mathcal{G}_{\bar x}$. Moreover,
one sees that in such case it is $\stsl{g}(\bar x)=\epsilon$. Thus
for some $r>0$ it must hold
$$
   f(x)+g(x)=f(x)-\epsilon d(x,\bar x)\ge f(\bar x)+g(\bar x),\quad
   \forall x\in\ball{\bar x}{r},
$$
which allows one to conclude that $\bar x$ is a local sharp minimizer
of $f$. This completes the proof.
\hfill$\square$
\end{proof}

\begin{remark}   \label{rem:shaminsta}
It is worth noting that the proof of Proposition \ref{pro:shamincharms}
actually reveals that if condition $(\mathcal{C})$ is valid for $f$
at $\bar x$, it continues being valid for any perturbed function 
$f+g$ at the same point,
for every $g\in\mathcal{G}_{\bar x}$, with  $\stsl{g}(\bar x)<
\grsl{f}(\bar x)$. In other words, sharp minimality itself is stable under
this kind of perturbation.
\end{remark}

%%%%%%%%%%%%%%%%%%%%%%%%%%%%%%%%%%%%%%%%%%%%%%%%%%%%

\subsection{Sufficiency in optimality conditions}
\label{subsec:2.2}

We have seen that
the steepest descent rate enables one to express a sufficient
condition for local optimality. The next proposition shows how the
same notion can be employed in formulating a sufficient condition
for the (global) solution existence.

\begin{proposition}    \label{pro:existsufcon}
Let $(X,d)$ be a complete metric space and let $f:X\longrightarrow\R
\cup\{+\infty\}$ be a l.s.c. function with $\inf_{X}f>-\infty$. If
there exists $\sigma>0$ such that
\begin{equation}     \label{in:hypthm}
    \sup_{x\in\levsupp{f}{\inf_{X}f}\cap\dom f}\grsl{f}(x)<-\sigma,
\end{equation}
then $\Argmin(f)\ne\varnothing$.
\end{proposition}

\begin{proof}
If $f\equiv+\infty$ the thesis is trivially true as it is $\Argmin(f)
=X$. Otherwise,
take an element $x_0\in X$ such that $f(x_0)<\inf_{X}f+\sigma$.
Since $f$ is l.s.c. and bounded from below, and $X$ is metrically
complete by hypothesis, it is possible to invoke the Ekeland
variational principle.
According to it, there exists $\bar x\in\ball{x_0}{1}$ such that
$f(\bar x)\le f(x_0)$, so $\bar x\in\dom f$, and
\begin{equation}      \label{in:EVP3}
    f(\bar x)<f(x)+\sigma d(x,\bar x),\quad\forall x\in X\backslash
   \{\bar x\}.
\end{equation}
Suppose now that $\bar x\in\levsupp{f}{\inf_{X}f}$. Then one finds
as a consequence of inequality $(\ref{in:EVP3})$ that
$\grsl{f}(\bar x)\ge -\sigma$, which contradicts the hypothesis
$(\ref{in:hypthm})$. Therefore it must be $\bar x\in\lev{f}{\inf_Xf}$,
so $\bar x$ turns out to be a global minimizer of $f$.
\hfill$\square$
\end{proof}

%%%%%%%%%%%%%%%%%%%%%%%%%%%%%%%%%%%%%%%%%%%%%%%%%%%%

\subsection{Error bound for inequalities}
\label{subsec:2.3}

The notion of local/global error bound is known to play a key role
in optimization and variational analysis.
Among other topics, it emerges as a crucial concept
in deriving exact penalty functions of constrained optimization
problems (see \cite{FacPan03}, Ch. 6.8) as well as in connection with
the property of calmness (equivalently, metric subregularity)
(see, for instance, \cite{FaHeKrOu10,Peno13,Krug14}).

\begin{definition}
Given a function  $f:X\longrightarrow\R\cup\{\pm\infty\}$ and an
element $\bar x\in X$, with $f(\bar x)=0$, $f$ is said to admit a
{\it local error bound} at $\bar x$ if there exist reals $c>0$
and $r>0$ such that
\begin{equation}    \label{in:locerbodef}
     \dist{x}{\lev{f}{0}}\le c[f(x)]_+,\quad\forall x\in\ball{\bar x}{r}.
\end{equation}
If inequality $(\ref{in:locerbodef})$ continues to hold with
$\ball{\bar x}{r}$ replaced by $X$, $f$ is said to admit a
{\it global error bound} at $\bar x$
\end{definition}

\begin{remark}
As done for instance in \cite{KrVaTh10},
it is worth observing that the best (lower) bound of all
constants $c$ for which inequality $(\ref{in:locerbodef})$
is true coincides with the value $(\Err{f}(\bar x))^{-1}$, where
$$
    \Err{f}(\bar x)=\liminf_{x\to\bar x\atop f(x)>0}
    {f(x)\over\dist{x}{\lev{f}{0}}}
$$
is  called the error bound modulus (aka conditioning rate)
of $f$ at $\bar x$.
\end{remark}

\begin{proposition}      \label{pro:erboCcon}
Let $f:X\longrightarrow\R\cup\{\pm\infty\}$ and $\bar x\in\dom f$
be such that $f(\bar x)=0$. If condition $(\mathcal{C})$ holds, then
$f$ admits a local error bound at $\bar x$. Moreover, it holds
$$
    \grsl{f}(\bar x)\le\Err{f}(\bar x).
$$
\end{proposition}

\begin{proof}
By virtue of condition $(\mathcal{C})$, $\bar x$ is a strict local
minimizer of $f$. More precisely, for every $\epsilon\in (0,\grsl{f}
(\bar x))$, there exists $\delta_\epsilon>0$ such that
\begin{equation}     \label{in:usecondC}
    f(x)\ge (\grsl{f}(\bar x)-\epsilon)d(x,\bar x)>0,\quad\forall
    x\in\ball{\bar x}{\delta_\epsilon}\backslash\{\bar x\}.
\end{equation}
This entails that
$$
    \lev{f}{0}\cap\ball{\bar x}{\delta_\epsilon}=\{\bar x\}.
$$
Consequently, one finds
$$
    \dist{x}{ \lev{f}{0}}=d(x,\bar x),\quad\forall x\in
    \ball{\bar x}{\delta_\epsilon/2},
$$
whence, taking account of inequality $(\ref{in:usecondC})$, it
readily follows
$$
    \dist{x}{\lev{f}{0}}\le (\grsl{f}(\bar x)-\epsilon)^{-1}[f(x)]_+,
    \quad\forall x\in\ball{\bar x}{\delta_\epsilon/2}.
$$
This proves that $f$ admits a local error bound at $\bar x$.
Besides, since $(\Err{f}(\bar x))^{-1}$ is the lower bound of all
constants $c$ satisfying inequality $(\ref{in:locerbodef})$,
it follows
$$
     \grsl{f}(\bar x)-\epsilon\le\Err{f}(\bar x).
$$
The arbitrariness of $\epsilon$ allows one to conclude the proof.
\hfill$\square$
\end{proof}

Throught simple counterexamples, one quickly realizes that the
local error bound property can take place even if $\grsl{f}(\bar x)
=0$, namely condition $(\mathcal{C})$ is in general only sufficient
for it. Nonetheless, if $\bar x$ strictly minimizes $f$, condition
$(\mathcal{C})$ becomes also necessary.

\begin{proposition}
Let $f:X\longrightarrow\R\cup\{\pm\infty\}$ and $\bar x\in\dom f$
be such that $f(\bar x)=0$. If $\bar x$ is a local strict minimizer
of $f$ and $f$ admits a local error bound at $\bar x$, then
condition $(\mathcal{C})$ holds true.
\end{proposition}

\begin{proof}
By the local error bound assumption at $\bar x$, one has that
for some $\delta$, $c>0$ it is
$$
   \dist{x}{\lev{f}{0}}\le c[f(x)]_+,\quad\forall x\in\ball{\bar x}{\delta}.
$$
Since $f(\bar x)=0$ and $\bar x$ is a strict local minimizer of $f$,
by a proper reduction of the value of $\delta$, one has
$$
   f(x)>0,\quad\forall x\in\ball{\bar x}{\delta}.
$$
Consequently, it results in
$$
   [f(x)]_+=f(x) \qquad\hbox{and}\qquad\dist{x}{\lev{f}{0}}
   =d(x,\bar x),\qquad\forall x\in\ball{\bar x}{\delta/2}.
$$
It follows
$$
    \inf_{x\in\ball{\bar x}{\delta/2}\backslash\{\bar x\}}
   {f(x)-f(\bar x)\over d(x,\bar x)}\ge {1\over c},
$$
whence one obtains the thesis.
\hfill$\square$
\end{proof}

While in general condition $(\mathcal{C})$ can not characterize
the validity of the local error bound property for $f$ at $\bar x$,
it enables one to single out a stronger property than that: actually,
it guarantees the local error bound property to hold for the whole family
of functions, which are sufficiently small perturbations of $f$, in
a sense clarified by the next proposition.

\begin{corollary}     \label{cor:erbosta}
Let $f:X\longrightarrow\R\cup\{\pm\infty\}$ and $\bar x\in\dom f$
be such that $f(\bar x)=0$. If condition $(\mathcal{C})$ holds, then
$f+g$ admits a local error bound at $\bar x$, for every $g\in
\mathcal{G}_{\bar x}$ such that $g(\bar x)=0$ and $\stsl{g}(\bar x)
<\grsl{f}(\bar x)$, and it results in
$$
    \grsl{f}(\bar x)-\stsl{g}(\bar x)\le\Err{(f+g)}(\bar x).
$$
\end{corollary}

\begin{proof}
As noticed in Remark \ref{rem:shaminsta}, condition $(\mathcal{C})$
is stable under additive perturbation of $f$ by $g\in
\mathcal{G}_{\bar x}$, provided that $\stsl{g}(\bar x)<\grsl{f}(\bar x)$.
Then, it suffices to apply Proposition \ref{pro:erboCcon} to function
$f+g$ at $\bar x$ and to recall that $\grsl{(f+g)}(\bar x)\ge\grsl{f}(\bar x)
-\stsl{g}(\bar x)$.
\hfill$\square$
\end{proof}

\begin{remark}        \label{rem:erbosta}
It is worth noting that, under the hypotheses of Corollary \ref{cor:erbosta},
the local error bound at $\bar x$ remains in force for every $\tilde f\in
\Ptb{f}{\bar x}{\epsilon}$, with $\epsilon<\grsl{f}(\bar x)$.
\end{remark}

%%%%%%%%%%%%%%%%%%%%%%%%%%%%%%%%%%%%%%%%%%%%%%%%%%%%
%%%%%%%%%%%%%%%%%%%%%%%%%%%%%%%%%%%%%%%%%%%%%%%%%%%%

\section{Subdifferential characterizations of condition $(\mathcal{C})$}
\label{sec:3}

In what follows, given a metric space $(X,d)$, ${\rm Lip}(X)$
will denote the vector space of all Lipschitz continuous functionals
defined on $X$, equipped with the quasinorm
$$
   \|\phi\|_{\rm Lip}=\sup_{x_1,x_2\in X\atop x_1\ne x_2}
  {|\phi(x_1)-\phi(x_2)|\over d(x_1,x_2)},\quad\phi\in {\rm Lip}(X),
$$
which induces the quasimetric $d_{\rm Lip}:{\rm Lip}(X)\times
{\rm Lip}(X)\longrightarrow [0,+\infty)$. If introducing the equivalence
relation over ${\rm Lip}(X)$
$$
   \phi_1\sim\phi_2\qquad \hbox{ iff }\qquad
  \exists c\in\R:\ \phi_1(x)-\phi_2(x)=c,\quad\forall x\in X,
$$
then $\|\cdot\|_{\rm Lip}$ is well defined on equivalence classes
and $({\rm Lip}(X)/_\sim,\|\cdot\|_{\rm Lip})$ turns out to be
a Banach space (see \cite{PalRol97}). Once fixed a nonempty
family $\Fam\subseteq {\rm Lip}(X)/_\sim$, it is possible to introduce
related concepts of subgradients and subdifferentials of functions
defined on $X$, as done for instance in \cite{PalRol97} (see also
references therein).
For the sake of notational simplicity, elements in ${\rm Lip}(X)/_\sim$
will be henceforth indicated with the same symbols as their
representatives in ${\rm Lip}(X)$.

\begin{definition}     \label{def:Famsubds}
Given a function $f:X\longrightarrow\R\cup\{\pm\infty\}$, $\bar
x\in\dom f$, and $\Fam\subseteq {\rm Lip}(X)/_\sim$, the set
$$
    \Fsubd f(\bar x)=\{\phi\in\Fam:\ f(x)-f(\bar x)\ge\phi(x)-\phi(\bar x),
    \quad\forall x\in X\}
$$
is called the {\it $\Fam$-subdifferential} of $f$ at $\bar x$; again, the set
$$
    \Fsubdloc f(\bar x)=\bigcap_{\epsilon>0}
    \Fsubdeps{\epsilon} f(\bar x),
$$
where, for a given $\epsilon\ge 0$, it is
\begin{eqnarray*}
   \Fsubdeps{\epsilon} f(\bar x) &=& \{\phi\in\Fam:\ \exists r>0:\ f(x)-f(\bar x)
    \ge\phi(x)-\phi(\bar x)-\epsilon d(x,\bar x),  \\
    & &\quad\forall x\in\ball{\bar x}{r}\},
\end{eqnarray*}
is called the {\it local $\Fam$-subdifferential} of $f$ at $\bar x$.
\end{definition}

\begin{remark}
Take into account that the inequalities defining the $\Fam$-subdifferential
and its local counterpart involve an abuse of notation, because
$\phi$ should indicate a class and not a single function. Nevertheless,
since $\phi_1\sim\phi_2$ implies $\phi_1(x)-\phi_1(\bar x)=\phi_2(x)-\phi_2(\bar x)$
for every $x\in X$, such inequalities are well defined.
\end{remark}

From Definition \ref{def:Famsubds} it is readily seen that
$$
    \Fsubd f(\bar x)\subseteq \Fsubdeps{0} f(\bar x)\subseteq
   \Fsubdloc f(\bar x)\subseteq \Fsubdeps{\epsilon} f(\bar x),
   \quad\forall\epsilon>0.
$$
Notice that, if equipped with the metric $d_{\rm Lip}$, $\Fam$
becomes a metric space. In the next proposition the topological
notion of interior refers to such a metric structure on $\Phi$.
The null element of ${\rm Lip}(X)$ and its $\sim$-equivalent
class is denoted here by $\nullv$.

\begin{proposition}     \label{pro:subdnecconC}  
Given a function $f:X\longrightarrow\R\cup\{\pm\infty\}$ and $\bar x
\in\dom f$, let $\Fam\subseteq {\rm Lip}(X)/_\sim$ be such that
$\nullv\in\Fam$. If condition $(\mathcal{C})$ holds then
$$
   \nullv\in\inte \Fsubdeps{0} f(\bar x).
$$
In particular, for every $\sigma\in (0,\grsl{f}(\bar x))$, one has
$\ball{\nullv}{\sigma}\subseteq\Fsubdeps{0} f(\bar x)$.
\end{proposition}

\begin{proof}
According to condition $(\mathcal{C})$, fixing an arbitrary
$\sigma\in (0,\grsl{f}(\bar x))$, there exists $r>0$ such that
\begin{equation}
    {f(x)-f(\bar x)\over d(x,\bar x)}\ge\sigma,\quad\forall
    x\in\ball{\bar x}{r}.      \label{in:propCeffect}
\end{equation}
So, take an arbitrary $\phi\in\ball{\nullv}{\sigma}$. Recalling the
definition of  $d_{\rm Lip}$, one has in particular
$$
    {|\phi(x)-\phi(\bar x)|\over d(x,\bar x)}\le\sigma,\quad\forall
    X\backslash\{\bar x\},
$$
whence, owing to inequality $(\ref{in:propCeffect})$, it is
$$
     f(x)-f(\bar x)\ge\phi(x)-\phi(\bar x),\quad\forall x\in
    \ball{\bar x}{r}.
$$
The last inequality shows that $\phi\in \Fsubdeps{0} f(\bar x)$
and such an inclusion implies that $\ball{\nullv}{\sigma}\subseteq
\Fsubdeps{0} f(\bar x)$.
\hfill$\square$
\end{proof}

The reader should observe that the above necessary condition holds
upon a  rather general assumption on the class $\Fam$. It is not
difficult to realize, through proper choices of $\Fam$, that without
additional assumptions the assertion of  Proposition \ref{pro:subdnecconC}  
can not be reversed. The next definition is aimed at introducing some
additional assumptions on $\Fam$ that seem to work in order to formulate
sufficient conditions for $(\mathcal{C})$.

\begin{definition}     \label{def:supdispro}
Let $\Fam\subseteq {\rm Lip}(X)/_\sim$ be a family containing
$\nullv$. $\Fam$ is said to satisfy the

(i) {\it supporting distance property} if there exists $\kappa\in
(0,+\infty)$ such that for every $\epsilon>0$
\begin{equation}       \label{in:supdistpro}
    \sup_{\phi\in\ball{\nullv}{\kappa\epsilon}}[\phi(x)-\phi(\bar x)]
    \ge\epsilon d(x,\bar x),\quad\forall x,\, \bar x\in X; 
\end{equation}

(ii) {\it local supporting distance property} at $\bar x$ if there
exists $\kappa\in (0,+\infty)$ such that for every $\epsilon>0$
there is $r>0$ such that
\begin{equation}      \label{in:supdistproloc}
    \sup_{\phi\in\ball{\nullv}{\kappa\epsilon}}[\phi(x)-\phi(\bar x)]
    \ge\epsilon d(x,\bar x),\quad\forall x\in\ball{\bar x}{r}.
\end{equation}
\end{definition}

Clearly, property (i) implies property (ii) in Definition \ref{def:supdispro}.
Moreover, if $\Phi$ is a class fulfilling (i) or (ii) and $\tilde\Phi$
is any other class such that $\tilde\Phi\supseteq\Phi$, then also
$\tilde\Phi$ does.
Some examples of families in ${\rm Lip}(X)/_\sim$ fulfilling the above
properties will be provided in the next section (see Remark
\ref{rem:supdisprodual}).

\begin{proposition}     \label{pro:subdsuf1conC}  
Given a function $f:X\longrightarrow\R\cup\{\pm\infty\}$ and $\bar x
\in\dom f$, let $\Fam\subseteq {\rm Lip}(X)/_\sim$. If $\Fam$ satisfies
the supporting distance property, then $\nullv\in\inte \Fsubd  f(\bar x)$
implies condition $(\mathcal{C})$.
\end{proposition}

\begin{proof}
By hypothesis there exists $\epsilon>0$ such that $\ball{\nullv}
{\epsilon}\subseteq\Fsubd  f(\bar x)$. By virtue of the supporting
distance property, there exists $\kappa\in (0,+\infty)$ such that
inequality $(\ref{in:supdistpro})$ holds true. Consequently,
one finds
\begin{eqnarray*}
   f(x)-f(\bar x) &\ge & \sup_{\phi\in\Fsubd  f(\bar x)}[\phi(x)-\phi(\bar x)]
    \ge\sup_{\phi\in\ball{\nullv}{\epsilon}}[\phi(x)-\phi(\bar x)] \\
      & \ge & {\epsilon\over\kappa}d(x,\bar x),\quad\forall x\in\ X,
\end{eqnarray*}
whence
\begin{equation}     \label{in:gloepskappa}
    {f(x)-f(\bar x)\over d(x,\bar x)}\ge {\epsilon\over\kappa},
    \quad\forall  x\in X\backslash\{\bar x\}.
\end{equation}
From the last inequality one immediately gets the validity of
condition $(\mathcal{C})$.
\hfill$\square$
\end{proof}

Localizing the notion of subdifferential as well as the supporting
distance property allows one to obtain a milder sufficient condition.
Nevertheless, the price to be paid for such a refinement of the
preceding result is an extra compactness assumption to be taken,
which limits the range of application.

\begin{proposition}     \label{pro:subdsuf2conC}  
Given a function $f:X\longrightarrow\R\cup\{\pm\infty\}$ and $\bar x
\in\dom f$, let $\Fam\subseteq {\rm Lip}(X)/_\sim$. If $\Fam$ satisfies
the local supporting distance property at $\bar x$ and balls in
$\Fam$ are compact, then $\nullv\in\inte \Fsubdloc f(\bar x)$
implies condition $(\mathcal{C})$.
\end{proposition}

\begin{proof}
Assume that $\ball{\nullv}{\sigma}\subseteq\Fsubdloc f(\bar x)$,
for some $\sigma>0$. By using the supporting distance property
at $\bar x$, one gets the existence of $\kappa>0$ as in (ii) of 
Definition \ref{def:supdispro}. Notice that one can assume that
$\kappa\ge 1$. Take an arbitrary $\phi_0\in\ball{\nullv}{\sigma}$.
Since $\phi_0$ belongs in particular to
$\Fsubdeps{{\sigma\over 4\kappa}}f(\bar x)$, then there exists
$r_{\sigma,\phi_0}>0$ such that
$$
    f(x)-f(\bar x)\ge \phi_0(x)-\phi_0(\bar x)-{\sigma\over 4\kappa}
    d(x,\bar x),\quad\forall x\in\ball{\bar x}{r_{\sigma,\phi_0}},
$$
so that
$$
   \inf_{x\in\ball{\bar x}{r_{\sigma,\phi_0}}\backslash\{\bar x\}}
  {f(x)-f(\bar x)-[\phi_0(x)-\phi_0(\bar x)]\over d(x,\bar x)}\ge
  -{\sigma\over 4\kappa}.
$$
Without loss of generality it is possible to assume that
$r_{\sigma,\phi_0}<\min\{r,\sigma/4\kappa\}$, where $r$ is
as in inequality $(\ref{in:supdistproloc})$, corresponding to
$\epsilon=\sigma/\kappa$. Now, observe that
for every $\phi\in\ball{\phi_0}{r_{\sigma,\phi_0}}$ it results in
\begin{eqnarray*}
    \inf_{x\in\ball{\bar x}{r_{\sigma,\phi_0}}\backslash\{\bar x\}} &
  \displaystyle{f(x)-f(\bar x)-[\phi(x)-\phi(\bar x)]\over d(x,\bar x)} \ge \\
   \inf_{x\in\ball{\bar x}{r_{\sigma,\phi_0}}\backslash\{\bar x\}} &
  \displaystyle {f(x)-f(\bar x)-[\phi_0(x)-\phi_0(\bar x)]\over d(x,\bar x)}- \\ 
  \sup_{x\in\ball{\bar x}{r_{\sigma,\phi_0}}\backslash\{\bar x\}} &
  \displaystyle{\phi(x)-\phi(\bar x)-[\phi_0(x)-\phi_0(\bar x)]\over d(x,\bar x)} \ge  \\
   \inf_{x\in\ball{\bar x}{r_{\sigma,\phi_0}}\backslash\{\bar x\}} &
  \displaystyle {f(x)-f(\bar x)-[\phi_0(x)-\phi_0(\bar x)]\over d(x,\bar x)}- \\ 
  \sup_{x_1,x_2\in X\atop x_1\ne x_2} &
  \displaystyle{\phi(x_1)-\phi(x_2)-[\phi_0(x_1)-\phi_0(x_2)]\over d(x_1,x_2)} \ge \\
   & \displaystyle -{\sigma\over 4\kappa}-d_{\rm Lip}(\phi,\phi_0)\ge
    -{\sigma\over 2\kappa},
\end{eqnarray*}
wherefrom it follows
$$
   \inf_{\phi\in\ball{\phi_0}{r_{\sigma,\phi_0}}}
   \inf_{x\in\ball{\bar x}{r_{\sigma,\phi_0}}\backslash\{\bar x\}}
   {f(x)-f(\bar x)-[\phi(x)-\phi(\bar x)]\over d(x,\bar x)}\ge 
   -{\sigma\over 2\kappa}.
$$ 
The family $\{\inte\ball{\phi_0}{r_{\sigma,\phi_0}}:\ \phi_0\in
\ball{\nullv}{\sigma}\}$ forms an open covering
of $\ball{\nullv}{\sigma}$, which is a compact set by
hypothesis. Therefore, there exist $N\in\N$ and $\phi_1$,
\dots ,$\phi_N\in\ball{\nullv}{\sigma}$ such that the subfamily
$\{\inte\ball{\phi_i}{r_{\sigma,\phi_i}}:\ i=1,\dots,N\}$ is still a
covering for $\ball{\nullv}{\sigma}$. Define
$$
    r_0=\min_{i=1,\dots,N}r_{\sigma,\phi_i}.
$$
Since for every $\phi\in\ball{\nullv}{\sigma}$ an
index $i_*\in\{1,\dots,N\}$ can be found such that $\phi\in
\ball{\phi_{i_*}}{r_{\sigma,\phi_{i_*}}}$, then by virtue of the
supporting distance property at $\bar x$, recalling that
$r_0<r$, one obtains
\begin{eqnarray*}
  -{\sigma\over 2\kappa} & \le &  \inf_{\phi\in\ball{\nullv}{\sigma}}
  \inf_{x\in\ball{\bar x}{r_0}\backslash\{\bar x\}}
  {f(x)-f(\bar x)-[\phi(x)-\phi(\bar x)]\over d(x,\bar x)}  \\
    & \le & \inf_{x\in\ball{\bar x}{r_0}\backslash\{\bar x\}}
   \left[{f(x)-f(\bar x)\over d(x,\bar x)}-\sup_{\phi\in\ball{\nullv}{\sigma}}
  {\phi(x)-\phi(\bar x)\over d(x,\bar x)}\right]  \\
   & \le & \inf_{x\in\ball{\bar x}{r_0}\backslash\{\bar x\}}
   {f(x)-f(\bar x)\over d(x,\bar x)}-\inf_{x\in\ball{\bar x}{r_0}\backslash\{\bar x\}}
  \sup_{\phi\in\ball{\nullv}{\sigma}}{\phi(x)-\phi(\bar x)\over d(x,\bar x)} \\
  & \le & \grsl{f}(\bar x)-{\sigma\over\kappa}.
\end{eqnarray*}
Thus, it is $ \grsl{f}(\bar x)\ge\sigma/2\kappa$. This completes
the proof.
\hfill$\square$
\end{proof}

Combining Proposition \ref{pro:subdnecconC} and Proposition
\ref{pro:subdsuf2conC} puts one in a position to derive the following
characterization of condition $(\mathcal{C})$ in subdifferential terms.

\begin{corollary}     \label{cor:subdcha1conC}  
Given a function $f:X\longrightarrow\R\cup\{\pm\infty\}$ and $\bar x
\in\dom f$, let $\Fam\subseteq {\rm Lip}(X)/_\sim$. Suppose that
$\Fam$ satisfies the local supporting distance property at $\bar x$
and balls in $\Fam$ are compact. Then condition $(\mathcal{C})$
holds iff
$$
    \nullv\in\inte \Fsubdloc f(\bar x).
$$
\end{corollary}

\begin{proof}
The necessary part of the thesis follows from Proposition \ref{pro:subdnecconC},
after recalling that $\Fsubdeps{0} f(\bar x)\subseteq\Fsubdloc f(\bar x)$.
The sufficient one comes from Proposition \ref{pro:subdsuf2conC}.
\hfill$\square$
\end{proof}

For functions enjoiying a certain convexity property, it becomes
possible to drop out the compactness assumption on the balls and to use
the $\Phi$-subdifferential for characterizing condition $(\mathcal{C})$.

\begin{definition}
Let $\Fam\subseteq {\rm Lip}(X)/_\sim$ be a given family.
A function $f:X\longrightarrow\R\cup\{\pm\infty\}$ is said to be
{\it $\Phi$-convex} if
$$
     f(x)=\sup\{\ell(x):\ [\ell]_\sim\in\Fam,\ \ell\le f\},\quad\forall x\in X,
$$
where $\ell\le f$ means that $\ell(x)\le f(x)$ for every $x\in X$.
\end{definition}

\begin{remark}     \label{rem:locglosubdif}
Whenever $X$ is a vector space, for certain families $\Fam\subseteq
{\rm Lip}(X)/_\sim$ it happens that if the inequality $\phi(x)\le
f(x)$ holds in a neighbourhood of a point $\bar x\in X$, then it
continues being valid on the whole space $X$. This is the case,
for example, of $\Phi$ given by the linear functionals on $X$ and
the classic concept of convexity. Notice that the property
\begin{equation}    \label{in:locglopro}
   \phi(x)\le f(x),\quad\forall  x\in\ball{\bar x}{r}\qquad
   \hbox{implies}\qquad \phi(x)\le f(x),\quad\forall  x\in X,
\end{equation}
entails that, if $f$ is $\Fam$-convex, then
$$
   \Fsubd f(\bar x)=\Fsubdeps{0} f(\bar x).
$$
\end{remark}

\begin{corollary}    \label{cor:Famconchar}
Given a function $f:X\longrightarrow\R\cup\{\pm\infty\}$ and
$\bar x\in\dom f$, let $\Fam\subseteq{\rm Lip}(X)/_\sim$ satisfy
the supporting distance property and  property $(\ref{in:locglopro})$.
Suppose that $f$ is $\Fam$-convex. Then condition $(\mathcal{C})$
holds iff $\nullv\in\inte\Fsubd f(\bar x)$.
\end{corollary}

\begin{proof}
In the light of Proposition \ref{pro:subdnecconC} and Proposition
\ref{pro:subdsuf1conC}, the thesis becomes an obvious consequence
of Remark \ref{rem:locglosubdif}.
\hfill$\square$
\end{proof}

%%%%%%%%%%%%%%%%%%%%%%%%%%%%%%%%%%%%%%%%%%%%%%%%%%%%
%%%%%%%%%%%%%%%%%%%%%%%%%%%%%%%%%%%%%%%%%%%%%%%%%%%%

\section{Consequences on stability in variational analysis}
\label{sec:4}

Throughout the present section, $(\X,\|\cdot\|)$ denotes a real
Banach space, with null vector $\nullv$. Its topological dual is
marked by $\X^*$, whose null vector is $\nullv^*$,
whereas the duality pairing $\X^*$ and $\X$ is indicated by
$\langle\cdot,\cdot\rangle$. Set $\Uball=\ball{\nullv}{1}$ and $\Usfer=
\{u\in\Uball:\ \|u\|=1\}$ and, similarly, $\Uball^*=\ball{\nullv^*}{1}$ and
$\Usfer^*=\{u\in\Uball^*:\ \|u\|=1\}$.
Given a function $f:\X\longrightarrow\R\cup\{\pm\infty\}$ and $\bar x
\in\dom f$, the Fr\'echet derivative of $f$ at $\bar x$ is denoted by
$\FrDer f(\bar x)$.

Below some situations are illustrated, in which the conditions
involving the steepest descent rate discussed in Section \ref{sec:2}
seem to be ``not natural" for smooth functions.

\begin{remark}
(i) If function $f:\X\longrightarrow\R\cup\{\pm\infty\}$ is Fr\'echet
differentiable at $\bar x\in\dom f$, then its steepest descent rate
at $\bar x$ can be represented as follows
\begin{equation}   \label{eq:grslFrdifrep}
     \grsl{f}(\bar x)=\inf_{u\in\Usfer}\langle\FrDer f(\bar x),u\rangle.
\end{equation}
Indeed, setting $o(\|x-\bar x\|)=f(x)-f(\bar x)-\langle\FrDer f(\bar x),
x-\bar x\rangle$, one finds
\begin{eqnarray*}
  \liminf_{x\to\bar x}{f(x)-f(\bar x)\over d(x,\bar x)} &=&
  \liminf_{x\to\bar x}\left[\left\langle \FrDer f(\bar x),{x-\bar x\over
   \|x-\bar x\|}\right\rangle+{o(\|x-\bar x\|)\over\|x-\bar x\|}\right] \\
  &=& \liminf_{x\to\bar x}\left\langle \FrDer f(\bar x),{x-\bar x\over
   \|x-\bar x\|}\right\rangle=\inf_{u\in\Usfer}\langle\FrDer f(\bar x),
  u\rangle.
\end{eqnarray*}
So, being $\FrDer f(\bar x)\in\X^*$, in such event it must be
$\grsl{f}(\bar x)\le 0$. This shows that condition $(\mathcal{C})$
can never be fulfilled by a function Fr\'echet differentiable at
$\bar x$.

(ii) Condition $(\ref{in:hypthm})$ can never be satisfied by a non
constant function $f\in C^1(\X)$ admitting global minimizers.
Indeed, in such case the set $\Argmin (f)\ne\X$ is closed.
Let $\bar x\in\bd\Argmin (f)$. Then there exists $(x_n)_{n\in\N}$,
with $x_n\to\bar x$ as $n\to\infty$ and $x_n\not\in\Argmin (f)$.
According to condition $(\ref{in:hypthm})$, for some $\sigma>0$
it should be 
\begin{equation}   \label{in:grslnegativ}
\grsl{f}(x_n)=\inf_{u\in\Usfer}|\langle\FrDer f(x_n),
u\rangle|<-\sigma,\quad\forall n\in\N.
\end{equation}
Since it is $\FrDer f(x_n)\to\nullv^*=\FrDer f(\bar x)$ as $n\to\infty$
by continuity of the mapping $\FrDer f:\X\longrightarrow\X^*$,
there must exist $n_\sigma\in\N$ such that
$$
    \sup_{u\in\Usfer}|\langle\FrDer f(x_n),u\rangle|=\|\FrDer f(x_n)\|
    \le{\sigma\over 2},\quad\forall n\in\N,\ n\ge n_\sigma,
$$
so one actually finds
$$
    \inf_{u\in\Usfer}\langle\FrDer f(x_n),u\rangle\ge -{\sigma\over 2},
   \quad\forall n\in\N,\ n\ge n_\sigma
$$
which contradicts evidently inequality $(\ref{in:grslnegativ})$.
\end{remark}

For a nonsmooth function $f:\X\longrightarrow\R\cup\{\pm\infty\}$
the representation $(\ref{eq:grslFrdifrep})$ is replaced by the 
side estimate
\begin{equation}     \label{in:grslnonsmoothest}
     \grsl{f}(\bar x)\le\inf_{u\in\Usfer} \HlDer f(\bar x;u),
\end{equation}
where
$$
    \HlDer f(\bar x;u)=\liminf_{v\to u\atop t\downarrow 0}
    {f(\bar x+tv)-f(\bar x)\over t}
$$
indicates the Hadamard lower derivative of $f$ at $\bar x$, in
the direction $u\in\X$ (see, for instance, \cite{Demy05b}).
While estimate $(\ref{in:grslnonsmoothest})$ generally
is not useful  to get conditions which are sufficient
for $({\mathcal C})$, it enables one to formulate the following
sufficient condition for the solution existence of nondifferentiable
optimization problems.

\begin{theorem}
Let $f:\X\longrightarrow\R\cup\{+\infty\}$ be a l.s.c. function
bounded from below. If there exists $\sigma>0$ such that
$$
    \sup_{x\in\levsupp{f}{\inf_{X}f}\cap\dom f}
    \inf_{u\in\Usfer} \HlDer f(x;u)<-\sigma,
$$
then $\Argmin(f)\ne\varnothing$.
\end{theorem}

\begin{proof}
The thesis is an obvious consequence of Proposition \ref{pro:existsufcon}
and of the estimate $(\ref{in:grslnonsmoothest})$.
\hfill$\square$
\end{proof}

Nevertheless, some special cases are known in which
inequality $(\ref{in:grslnonsmoothest})$ turns out to hold as an
equality.
For example, if $\X=\R^n$, by virtue of the compactness of balls,
it has been shown that 
$$
    \grsl{f}(\bar x)=\min_{u\in\Usfer}\HlDer f(\bar x;u)
$$
(see \cite{Demy05b,PapUde99}). In this case, condition $({\mathcal C})$
can be characterized via the positivity of $\HlDer f(\bar x;u)$ over
$\Usfer$.

On the other hand,
when working with nonsmooth functions defined on Banach spaces
widely tools of analysis are generalized derivative constructions
based on the dual space. In this concern, observe that if
$\mathfrak{A}(\X)\subseteq{\rm Lip}(\X)$ denotes the family consisting of all
affine functions on $X$, then $\mathfrak{A}(\X)/_\sim$ can be identified
with $\X^*$. Note that in such case $\|\cdot\|_{\rm Lip}$ becomes
the usual (uniform) norm in $\X^*$. By this choice of $\Fam$,
$\Fsubd f(\bar x)$ coincides with the subdifferential of $f$ at
$\bar x$ in the sense of convex analysis, here denoted simply
by $\partial f(\bar x)$, whereas $\Fsubdloc f(\bar x)$ coincides
with the regular (aka Fr\'echet) subdifferential of $f$ at $\bar x$,
here denoted by $\regsubd f(\bar x)$, i.e.
$$
     \regsubd f(\bar x)=\left\{x^*\in\X^*:\ \liminf_{x\to\bar x}
     {f(x)-f(\bar x)-\langle x^*,x-\bar x\rangle\over \|x-\bar x\|}
     \ge 0\right\}
$$
(for detailed expositions concerning this construction
see \cite{Mord06,Peno13,RocWet98,Schi07}).
Furthermore, whenever a l.s.c. function $f:\X\longrightarrow\R
\cup\{+\infty\}$ is $\mathfrak{A}(\X)/_\sim$-convex, then it is convex
in the classical sense.

\begin{remark}     \label{rem:supdisprodual}
By standard separation arguments of convex analysis, it is not
difficult to see that $\X^*$ satisfies the supporting distance
property. Indeed, by taking $\kappa=1$, one actually has
$$
   \sup_{x^*\in\epsilon\Uball^*}\langle x^*,x-\bar x\rangle=
   \epsilon\|x-\bar x\|,\quad\forall x,\bar x\in\X.
$$
Along with $\mathfrak{A}(\X)$, other subclasses of ${\rm Lip}(\X)$
satisfying the supporting distance property and leading to
interesting $\Phi$-subdifferential constructions are
$$
   \mathfrak{S}(\X)=\{\phi:\X\longrightarrow\R:\ \phi
   \hbox{ is sublinear and continuous on $\X$}\}
$$
$$
   \mathfrak{DS}(\X)=\mathfrak{S}(\X)-\mathfrak{S}(\X).
$$
\end{remark}

By specializing to a Banach space setting what established in
Section \ref{sec:3}, it is possible to extend and generalize the characterization
of sharp minimality presented in \cite{Poly87} (Ch. 5, Lemma 3).

\begin{theorem}    \label{thm:bsconvchar}
Let $f:\X\longrightarrow\R\cup\{+\infty\}$ be a l.s.c. convex function.
Suppose that $f$ is continuous at $\bar x\in\dom f$. Then condition
$(\mathcal{C})$ holds iff $\nullv^*\in\inte\partial f(\bar x)$.
\end{theorem}

\begin{proof}
On the base of the current subdifferential contructions, by virtue of
what noticed in Remark \ref{rem:supdisprodual}, it suffices to apply
Corollary \ref{cor:Famconchar}.
\hfill$\square$
\end{proof}

Since a convex function is G\^ateaux differentiable at a given
point $\bar x$ of its domain iff its subdifferential reduces to
a singleton, condition $(\mathcal{C})$ evidently appear to be inconsistent
with such kind of smoothness at $\bar x$. Below, some of the benefits
concerning the stability behaviour of nonsmooth functions
are listed.

\begin{theorem}     \label{thm:convbenefit}
Let $f:\X\longrightarrow\R\cup\{+\infty\}$ be a l.s.c. convex
function, which is continuous at $\bar x\in\dom f$. Suppose
that $f(\bar x)=0$. If $\nullv^*\in\inte\partial f(\bar x)$ the following assertions are true:

\noindent (i) $\bar x$ is a global sharp minimizer of $f$;

\noindent (ii) $\bar x$ is superstable;

\noindent (iii) function $f$ admits a global error bound at $\bar x$;

\noindent (iv) there exists $\epsilon>0$ such that every function
$\tilde f\in\Ptb{f}{\bar x}{\epsilon}$ admits a local error bound
at $\bar x$.
\end{theorem}

\begin{proof}
In the light of the characterization provided by Theorem \ref{thm:bsconvchar},
assertions (i), (ii) and (iii) follow from the positivity of the steepest
descent rate of $f$ at $\bar x$ (remember Proposition \ref{pro:Ccondshamin},
Proposition \ref{pro:shamincharms}, and Proposition \ref{pro:erboCcon}). 
Indeed, under the current assumptions, the sharp minimality of
$\bar x$ can be obtained in a global form on account of inequality
$(\ref{in:gloepskappa})$. Besides, globality of the error bound
property trivially takes place because one has
$$
    \lev{f}{0}=\{\bar x\}\quad\hbox{ and }\quad
    f(x)=[f(x)]_+,\quad\forall x\in X.
$$
As for assertion (iv), it suffices to recall Remark \ref{rem:erbosta}.
\hfill$\square$
\end{proof}

\begin{remark}
(i) Notice that assertion (iii) in Theorem \ref{thm:bsconvchar}
allows one to recover the implication ``\ $\nullv^*\in\inte\partial
f(\bar x)\quad\Rightarrow\quad \hbox{globall error bound at
$\bar x$}$\ " appearing, among other results, in Theorem 1 of
\cite{KrVaTh10}. In turn, such implication enables one
to derive the following well-known error bound condition
for a convex inequality
$$
    f'(\bar x;v)\ge\sigma\|v\|,\quad\forall v\in\X,
$$
where $f'(\bar x;v)$ denotes the directional derivative of $f$
at $\bar x$, in the direction $v$. Its sufficiency indeed comes
directly from the dual representation
$$
    f'(\bar x;v)=\max_{x^*\in\partial f(\bar x)}\langle x^*,v
    \rangle,\quad\forall v\in\X.
$$
The necessariness instead follows from the local sharp minimality
of $\bar x$, which makes the inequality
$$
    {f(\bar x+tv)-f(\bar x)\over t}\ge\sigma\|v\|
$$
true for $t>0$ small enough.

(ii) It is clear that condition $\nullv^*\in\inte\partial f(\bar x)$ is far
from being necessary for the error bound of $f$ at $\bar x$. Indeed, the
latter property can be achieved under the condition $\nullv^*
\not\in\partial f(\bar x)$ as well (see again Theorem 1 in
\cite{KrVaTh10}).
\end{remark}

In the finite-dimensional case it is possible to establish the
following counterpart of Theorem \ref{thm:bsconvchar},
which applies to nonconvex functions.

\begin{theorem}     \label{thm:bsregsubdchar}
Let $f:\R^n\longrightarrow\R\cup\{\pm\infty\}$ be a function
and $\bar x\in\dom f$. Then condition $(\mathcal{C})$ holds iff
$\nullv^*\in\inte\regsubd f(\bar x)$.
\end{theorem}

\begin{proof}
Upon the identification $\mathfrak{A}(\X)/_\sim\cong\X^*$, all hypotheses
of Corollary \ref{cor:subdcha1conC} happen to be fulfilled
(recall Remark \ref{rem:supdisprodual}). Thus the thesis becomes
an obvious consequence of it.
\hfill$\square$
\end{proof}

The consequences on the variational behaviour stability of
nonconvex nonsmooth functions are presented below.
Notice that in the absence of convexity assertions $(i)$ and
$(iii)$ lose their global validity.

\begin{theorem}      \label{thm:nonconvbenefit}
Let $f:\R^n\longrightarrow\R\cup\{\pm\infty\}$ be a function,
with $\bar x\in\dom f$. Suppose that $f(\bar x)=0$. If $\nullv^*
\in\inte\regsubd f(\bar x)$ the following assertions are true:

\noindent (i) $\bar x$ is a local sharp minimizer of $f$;

\noindent (ii) $\bar x$ is superstable;

\noindent (iii) function $f$ admits a local error bound at $\bar x$;

\noindent (iv) there exists $\epsilon>0$ such that every function
$\tilde f\in\Ptb{f}{\bar x}{\epsilon}$ admits a local error bound.
\end{theorem}

\begin{proof}
It is possible to argue as in the proof of Theorem \ref{thm:bsconvchar},
using instead the characterization provided by Theorem
\ref{thm:bsregsubdchar}.
\hfill$\square$
\end{proof}

\begin{remark}
According to \cite{PolRoc98} a point $\bar x\in\X$ is said to give a
tilt-stable local minimum of $f:\X\longrightarrow\R\cup\{\pm\infty\}$,
if $\bar x\in\dom f$ and there exists $\delta>0$ such that the set-valued
mapping $\mathcal{M}_{f,\delta}:\X^*\rightrightarrows\X$ defined by
$$
    \mathcal{M}_{f,\delta}(x^*)=\left\{y\in\X:\ y\hbox{ solves } 
    \min_{x\in\ball{\bar x}{\delta}} [f(x)-f(\bar x)-\langle x^*,
    x-\bar x\rangle]\right\}
$$
is single-valued and Lipschitz on some neighbourhood of $\nullv^*$,
with $\mathcal{M}_{f,\delta}(\nullv^*)=\{\bar x\}$. Since for any $x^*\in
\X^*$ and $x\in\X$, one has $\stsl{x^*}(x)=\|x^*\|$, from the superstability
of $\bar x$ for $f$ follows that, if taking $\delta<\grsl{f}(\bar x)$,
the mapping $\mathcal{M}_{f,\delta}$ in a neighbourhood
of $\nullv^*$ constantly takes the value $\{\bar x\}$. Thus, Theorem
\ref{thm:convbenefit} and Theorem \ref{thm:nonconvbenefit} entail
also the tilt-stability of $\bar x$.
\end{remark}

%%%%%%%%%%%%%%%%%%%%%%%%%%%%%%%%%%%%%%%%%%%%%%%%%%%%

\section{Conclusions}
\label{sec:5}

The meditations exposed in the present paper should convince a reader
that in optimization nonsmoothness does not mean necessarily
a pathology, leading only to handicaps in the problem analysis.
In certain situations having to do with the solution stability and
the sufficiency of optimality conditions,
nonsmoothness affords a robustness behaviour that smoothness
can not guarantee. Here some evidences of such a phenomenon
are collected and discussed. An aspect which seems
to be remarkable is that a unifying study of the issue can be
conducted already in a metric space setting, via a positivity
condition on the steepest descent rate.

%%%%%%%%%%%%%%%%%%%%%%%%%%%%%%%%%%%%%%%%%%%%%%%%%%%%

%\begin{acknowledgements}
%\end{acknowledgements}


\end{document}